\documentclass[11pt]{amsart}
\usepackage[a4paper, margin=28mm]{geometry} 
\usepackage{mathtools,amssymb,bm,amsthm,enumerate} 
\usepackage{iftex}
\ifPDFTeX
  \usepackage[utf8]{inputenc} 
  \usepackage[T1]{fontenc} 
  \usepackage{lmodern} 
\else
  \usepackage[no-math]{fontspec} 
\fi
\usepackage[english]{babel} 
\usepackage{xcolor} 
\usepackage{graphicx} 
\usepackage[pdfencoding=auto]{hyperref} 

\theoremstyle{plain}

  \newtheorem{theorem}{Theorem}[section]
  \newtheorem{lemma}[theorem]{Lemma}
  \newtheorem{proposition}[theorem]{Proposition}

\newcommand{\R}{\mathbb{R}}
\newcommand{\RP}{\mathbb{RP}}
\newcommand{\GG}{\mathbb{G}}
\newcommand{\FF}{\mathcal{F}}
\newcommand{\proj}{\mathbb{P}}

\DeclareMathOperator{\vp}{vp}
\DeclareMathOperator{\tr}{tr}
\DeclareMathOperator{\cov}{cov}
\DeclareMathOperator{\Jac}{Jac}
\DeclareMathOperator{\Hess}{Hess}
\DeclareMathOperator{\Proj}{Proj}
\DeclareMathOperator{\GL}{GL}
\DeclareMathOperator{\PGL}{PGL}
\DeclareMathOperator{\Id}{Id}
\DeclareMathOperator{\cl}{cl}
\DeclareMathOperator{\intr}{int}

\numberwithin{equation}{section}

\title[Projective transformations and volume product]{Projective transformations of convex bodies and volume product}
\author{F. Balacheff, G. Solanes \& K. Tzanev}

\thanks{The first author acknowledges support by the FSE/AEI/MICINN grant RYC-2016-19334  ``Local and global systolic geometry and topology''. The second author is supported by the Serra Hunter Programme and the AEI/MICINN Mar\'ia de Maeztu grant CEX2020-001084-M.  The first and the second authors acknowledge support by the FEDER/AEI/MICINN grant PID2021-125625NB-I00 ``Estructuras y Desigualdades Geom\'etricas Universales'' and the AGAUR grant 2021-SGR-01015. }

\begin{document}
\begin{abstract}
  In this paper we study certain variational aspects of the volume product functional restricted to the space of small projective deformations of a fixed convex body. In doing so, we provide a short proof of a theorem by Klartag: a strong version of the slicing conjecture implies the Mahler conjecture. We also exhibit an interesting family of critical convex bodies in dimension $2$ containing saddle points for this functional.
\end{abstract}
\maketitle

\section{Introduction}

The volume product of a convex body $K\subset \R^n$ whose interior contains the origin is the quantity $\vp(K)\coloneqq|K|\cdot|K^\circ|$. Here $|\cdot |$ denotes the standard Lebesgue measure in $\R^n$, and $K^\circ=\{y \in \R^n\mid \langle x,y\rangle=\sum_{i=1}^n x_i y_i \leq 1\, \, \forall x \in K\}$ is the polar body of $K$. While it is well known \cite{San49} that the volume product among convex bodies whose barycenter is centered at the origin is maximized only by ellipsoids, the lower bound is still unknown in dimension greater than $2$. More precisely, the Mahler conjecture claims that any convex body $K$ containing the origin in its interior should satisfy the following optimal inequality:
\[
  \vp(K)\geq \frac{(n+1)^{n+1}}{(n!)^2}.
\]
The conjectured lower bound is reached by any $n$-simplex whose vertices span $\R^n$ and sum up to zero. Mahler actually proved the inequality in dimension $n=2$ \cite{Mah39}, but so far the conjecture remains open for higher dimensions.

If we endow the space of convex bodies in $\R^n$ with the usual Hausdorff topology, the volume product becomes a continuous function over the space $\mathcal{K}_0^n$ of convex bodies in $\R^n$ containing the origin in their interior. A standard compactness argument then shows that the functional $\vp : \mathcal{K}_0^n \to (0,\infty)$ attains its minimum. According to \cite{San49} any local minimizer $K$ of the volume product satisfies that its barycenter $\GG_K$ and the barycenter $\GG_{K^\circ}$ of its polar coincide with the origin.
Now recall that the covariance matrix $\cov(K)$ of a convex body $K$ whose barycenter is centered at the origin is the symmetric matrix with the following entries:
\[
  \cov(K)_{ij}=\frac{1}{|K|} \int_K x_i x_jdx.
\]
In this paper we propose a short proof of the following result due to Klartag.

\begin{theorem}\cite{Kla18}\label{th:main}
  Let $K$ be a local minimizer of the volume product. Then
  \[
    \cov(K^\circ)\geq \frac{1}{(n+2)^2}\cdot\cov(K)^{-1}
  \]
  in the sense of symmetric matrices.
\end{theorem}
It follows from this result that a strong version of the slicing conjecture implies the Mahler conjecture. This strong slicing conjecture states that the homogeneous quantity $\det [\cov(K)]/|K|^2$ is maximized among convex bodies with barycenter at the origin by any $n$-simplex whose vertices span $\R^n$ and sum up to zero, see \cite{Kla18} for further details.

Klartag's proof of Theorem \ref{th:main} is based on a variational argument in the space of projective images of $K$ that makes substantial use of the Laplace transform. In this paper we will also use a variational argument in the space of projective images of $K$ in order to recover Theorem \ref{th:main}, but our computation will be considerably simpler.
More precisely, in section \ref{sec:var} we consider the functional $\FF_K(x,\xi)\coloneqq\vp\left((K^\circ+\xi)^\circ+x\right)$ which is well-defined for any convex body $K$ in $\mathcal{K}_0^n$ on an open neighborhood of the origin in $\mathbb{R}^n\times\mathbb{R}^n$. By \cite{San49} we will easily observe (Proposition \ref{prop:jac}) that the origin is a critical point of the functional $\FF_K$ if and only if $\GG_K=\GG_{K^\circ}=0$. Then an elementary computation (Proposition \ref{prop:hess}) will lead us to the following expression for the Hessian matrix at such a critical point:
\[
  \Hess_{(0,0)} \FF_K=(n+1) \vp(K) \left(\begin{array}{cc}(n+2) \cov(K^\circ) & -I_n \\ -I_n & (n+2)\cov(K) \end{array}\right).
\]
Finally we will check (Proposition \ref{prop:equiv}) that the above Hessian is positive semi-definite if and only if the symmetric matrix inequality $(n+2) \cdot \cov(K^\circ)\geq [(n+2)\cdot\cov(K)]^{-1}$ holds true. This implies Theorem \ref{th:main}.

Our functional $\FF_K$ is defined using special projective transformations of the convex body, which are obtained as specific combinations of translations of the polar body with translations of the body itself. In section \ref{sec:proj} we will prove that this functional captures all the relevant information about variational properties of the volume product on the space of projective images of a given convex body.
More precisely, we will first prove in Proposition \ref{prop:def} that once a suitable affine coordinate chart is fixed, the projective transformations close to the identity map transform the convex body $K$ into a new convex body which still contains the origin in its interior. We will then show in Proposition \ref{prop:proj} that the volume product of these small projective perturbations of our convex body factors through the map $\FF_K$.

Finally in section \ref{sec:pent} we will exhibit an interesting family of convex bodies in dimension $2$ that are critical points of the volume product on the space of their own projective images.
More precisely, we will show the following.

\begin{theorem}\label{th:pentagons}
  There exists a $1$-parameter deformation $\{P_q\}_{q\in [-1/\sqrt{2},1/\sqrt{2}]}$ of an equilateral triangle $\Delta=P_{-1/\sqrt{2}}$ by axially symmetric pentagons satisfying the following properties.
  \begin{enumerate}[i)]
  \item \label{item_critical}
  Every pentagon $P_q$ is a critical point of the volume product in its own projective orbit; more precisely, the function $\mathcal{F}_{P_q}$ has a critical point at the origin.

  \item \label{item_saddle}
  The pentagon $P_0$ is a saddle point; more precisely the Hessian matrix $\Hess_{(0,0)}\mathcal{F}_{P_0}$ is non-degenerate and non-definite.

  \item \label{item_counterexample}
  For $q$ close enough to $-1/\sqrt{2}$ we have
  \begin{equation}\label{trace inequality}
    \tr[\cov(P_q)\cov(P_q^\circ)]>\tr[\cov(\Delta)\cov(\Delta^\circ)]=1/8.
  \end{equation}
  \end{enumerate}
\end{theorem}

The main relevant property of this family is that part of these critical points are saddle points. The last property that there exists small deformations of $\Delta$ verifying \eqref{trace inequality}
is also remarkable for the following reason. Recall that the $2$-dimensional Euclidean unit ball $B$ also satisfies $\tr[\cov(B)\cov(B^\circ)]=1/8$, and that Kuperberg in a private communication to Klartag shows that the Euclidean ball of any dimension is a local maximizer of the functional $K\mapsto \tr[\cov(K)\cov(K^\circ)]$ in the class of centrally-symmetric convex bodies with sufficiently smooth boundary, see \cite[p.78]{Kla18}. At some point it was not clear if this upperbound was a global one, until it was disproved in \cite[Proposition 1.5]{Kla18} for large dimensions. Our family discards any non-symmetric analog of Kuperberg's local result near the $n$-simplex, even in low dimensions.

\noindent \textbf{Acknowledgements.} We would like to thank B. Klartag for valuable exchanges.

\section{A simple variational study of the volume product}\label{sec:var}

The purpose of this section is to give a short proof of Theorem \ref{th:main}.

Fix a convex body $K$ of $\R^n$ containing the origin in its interior, and consider the functional
\begin{eqnarray*}
  \mathcal{F}_K: \, \, U_K & \to     & (0,\infty)                            \\
  (x,\xi)                  & \mapsto & \vp\left((K^\circ+\xi)^\circ+x\right)
\end{eqnarray*}
which is well-defined in an open neighbourhood $U_K$ of $(0,0) \in \R^n \times \R^n$.

\begin{proposition}\label{prop:jac}
The Jacobian matrix of $\FF_K$ at the origin is the following:
\[
  \Jac_{(0,0)} \FF_K=-(n+1) \vp(K) \left(\begin{array}{cc} \mathbb{G}_{K^\circ} & \mathbb{G}_{K} \end{array}\right).
\]
\end{proposition}

Recall the barycenter of a convex body $A$ is defined by the formula $\GG_A\coloneqq|A|^{-1}\int_A x dx$.

\begin{proof}
Equivalently, we shall prove that for any $(v,\eta)\in \R^n\times \R^n$, the following holds:
\[
  d_{(0,0)} \, \FF_K(v,\eta)=-(n+1) \vp(K) (\langle{\mathbb G}_{K^\circ},v\rangle+ \langle{\mathbb G}_{K},\eta\rangle).
\]
First recall that Santal\'o proved in \cite{San49} the following formula:
\[
  \left. \frac{d}{dt}\right|_{0} |(K-tv)^\circ|=(n+1) \cdot |K^\circ| \cdot \left\langle {\mathbb G}_{K^\circ},v\right\rangle.
\]
Then we easily compute that
\[
  d_{(0,0)} \, \FF_K(v,0)=\left.\frac{d}{dt}\right|_0 \vp(K+tv)=- (n+1) \cdot \vp(K) \cdot \left\langle {\mathbb G}_{K^\circ},v\right\rangle.
\]
The equality $d_{(0,0)} \, \FF_K(0,\eta)=-(n+1) \cdot \vp(K) \cdot \langle{\mathbb G}_{K},\eta\rangle$ easily follows from the previous computation using that $\vp(K)=\vp(K^\circ)$.
\end{proof}
In particular the origin is a critical point of the functional $\FF_K$ if and only if $\GG_K=\GG_{K^\circ}=0$.

\begin{proposition}\label{prop:hess}
Suppose that the origin is a critical point of $\FF_K$. Then its Hessian matrix is the following:
\[
  \Hess_{(0,0)} \FF_K=(n+1) \vp(K) \left(\begin{array}{cc}(n+2) \cov(K^\circ) & -I_n \\ -I_n & (n+2)\cov(K) \end{array}\right).
\]
\end{proposition}

\begin{proof}
First observe that for any $v,w \in \R^n$
\begin{align*}
\Hess_{(0,0)} \, \FF_K[(v,0),(w,0)]&=\left.\frac{d}{dt}\right|_0 \left.\frac{d}{ds}\right|_0 \FF_K(tv+sw,0)\\
(\text{using the previous proposition})&=- (n+1) \cdot \left.\frac{d}{dt}\right|_0 \vp(K+tv) \cdot \left\langle {\mathbb G}_{(K+tv)^\circ},w\right\rangle\\
&=- (n+1) \cdot |K| \cdot \left\langle \left.\frac{d}{dt}\right|_0 [|(K+tv)^\circ| \cdot {\mathbb G}_{(K+tv)^\circ}],w\right\rangle.
\end{align*}
Now, using the radial function $\rho_A$ of a convex body $A \in \mathcal{K}_0^n$, the support function $h_{A^\circ}$ of its polar set, and the fact that $\rho_A=1/h_{A^\circ}$ we compute
\begin{align*}
  \left.\frac{d}{dt}\right|_0 [|(K+tv)^\circ| \cdot {\mathbb G}_{(K+tv)^\circ}]
    &= \left.\frac{d}{dt}\right|_0 \int_{(K+tv)^\circ} x dx\\
    &= \left.\frac{d}{dt}\right|_0 \int_{S^{n-1}} \int_0^{\rho_{(K+tv)^\circ}(u)}(ru) \, r^{n-1}dr du\\
    &= \frac{1}{n+1} \cdot \left.\frac{d}{dt}\right|_0 \int_{S^{n-1}} \rho_{(K+tv)^\circ}(u)^{n+1} u \, du\\
    &= \frac{1}{n+1} \cdot \left.\frac{d}{dt}\right|_0 \int_{S^{n-1}} \frac{u}{h_{K+tv}(u)^{n+1}} du\\
    &= \frac{1}{n+1} \cdot \left.\frac{d}{dt}\right|_0 \int_{S^{n-1}} \frac{1}{(h_K(u)+t\langle v,u \rangle)^{n+1}} \, udu\\
    &= -\int_{S^{n-1}} \frac{\langle v,u \rangle}{h_K(u)^{n+2}} \, udu.
\end{align*}
Therefore
\begin{align*}
  \Hess_{(0,0)} \, \FF_K((v,0),(w,0))
    &= (n+1) \cdot |K| \cdot \int_{S^{n-1}} \frac{\langle v,u \rangle \cdot \langle w,u \rangle}{h_K(u)^{n+2}} \, du\\
    &= (n+1) \cdot (n+2) \cdot |K| \cdot \int_{K^\circ} \langle v,x \rangle \cdot \langle w,x \rangle \, dx\\
    &= (n+1) \cdot \vp(K) \cdot (n+2) \cdot v^t \cov(K^\circ) \ w
\end{align*}
as claimed. The computation of the dual term
\[
  \Hess_{(0,0)} \, \FF_K((0,\xi),(0,\eta))=(n+1) \cdot \vp(K) \cdot (n+2) \cdot \xi^t \cov(K) \ \eta.
\]
for any $\xi,\eta \in \R^n$ is similar. It remains to compute the mixed term as follows:
\begin{align*}
  \Hess_{(0,0)} \, \FF_K((v,0),(0,\eta))
    &= \left.\frac{d}{ds}\right|_0 \left(|(K^\circ+s\eta)^\circ| \cdot \left.\frac{d}{dt}\right|_0 |((K^\circ+s\eta)^\circ +tv)^\circ|\right)\\
    &= -(n+1)\left.\frac{d}{ds}\right|_0 \left(\vp(K^\circ+s\eta) \cdot \langle {\mathbb G}_{K^\circ+s\eta},v\rangle \right)\\
  (\text{the origin being a critical point})
    &= -(n+1) \cdot \vp(K^\circ) \cdot \left.\frac{d}{ds}\right|_0 \langle {\mathbb G}_{K^\circ+s\eta},v\rangle \\
    &= -(n+1) \cdot \vp(K) \cdot \langle \eta,v\rangle.
\end{align*}
This completes the proof.
\end{proof}

\begin{proposition}\label{prop:equiv}
Suppose that the origin is a critical point of $\FF_K$. Then
\[
  \Hess_{(0,0)} \FF_K\geq 0 \Longleftrightarrow \cov(K^\circ)\geq \frac{1}{(n+2)^2}\cdot\cov(K)^{-1}.
\]
Moreover
\begin{equation}\label{expression det}
  \det(\Hess_{(0,0)} \FF_K)=(n+1)^n\vp(K)^{n}\det[(n+2)^2\cov(K^\circ)\cov(K)-I_n].
\end{equation}
\end{proposition}

As observed in the introduction, the first part of this statement implies Theorem \ref{th:main}. The second part will be useful in the last section.

\begin{proof}
Just remark that if $A$ and $B$ are two symmetric positive definite matrices, and if we set
\[
  M\coloneqq \left(\begin{array}{cc}A & -I_n \\ -I_n & B \end{array}\right) \ \ \text{and} \ \ P\coloneqq\left(\begin{array}{cc}I_n & A^{-1} \\ 0 & I_n \end{array}\right),
\]
then $P$ being invertible the following holds true:
\[
  M\geq 0 \ \ \Longleftrightarrow P^t \ M \ P=\left(\begin{array}{cc}A & 0 \\ 0 & B-A^{-1} \end{array}\right)\geq 0 \ \Longleftrightarrow B \geq A^{-1}.
\]
Since $\det P=1$,
\[
  \det M = \det \left(\begin{array}{cc}A & 0 \\ 0 & B-A^{-1} \end{array}\right)=\det(A)\det(B-A^{-1})=\det(AB-I_n).\qedhere
\]
\end{proof}

\section{Projective transformations and volume product}\label{sec:proj}

The purpose of this section is to show that the volume product on the space of projective images of a given convex body $K$ factors through the map $\FF_K$. Therefore the functional $\FF_K$ captures all the relevant information about variational properties of the volume product on the space of projective images of a given convex body.

First recall some material about projective transformations (also called homographies).
Consider the real projective space $\RP^n=(\R^{n+1}\setminus \{0\})/\R^*$.
We will make use of the projective coordinates $[x_0,\ldots,x_n]$ for points in $\RP^n$. Recall that a \textit{projective transformation} is a diffeomorphism of $\RP^n$ induced by a linear isomorphism of $\R^{n+1}$ as follows: given $A=(a_{ij}) \in \GL_{n+1}(\R)$ the induced projective transformation $\proj{A}$ is
\[
  \proj{A}[x_0,\ldots , x_n]=\left[\sum_{i=0}^n a_{0i}x_i,\ldots,\sum_{i=0}^n a_{ni}x_i\right].
\]
This allows to identify the group of projective transformations $\Proj(\RP^n)$ with the projective linear group $\PGL_n(\R)=\GL_{n+1}(\R)/(\R^\ast\cdot I_{n+1})$ and defines a natural topology on it. Here $I_{n+1}$ denotes the identity matrix.
We will use the \textit{affine coordinate chart}
\begin{eqnarray*}
  F: \R^n            & \longrightarrow & \RP^n\setminus H_0        \\
  y=(y_1,\ldots,y_n) & \longmapsto     & [1,y]=[1,y_1,\ldots,y_n],
\end{eqnarray*}
where $H_0=\{[x_0,x_1,\ldots,x_n] \mid x_0=0\}$. Its inverse map is $F^{-1}[x_0,\ldots,x_n]=(x_1/x_0,\ldots,x_n/x_0)$.

Given two nested open balls centered at the origin $B_1\varsubsetneq B_2$, we define the open subset of $\mathcal{K}_0^n$
\[
  \mathcal{V}\coloneqq\{K \in \mathcal{K}_0^n \mid \cl({B}_1) \subset \intr(K)\subset K \subset B_2\}
\]
where $\cl$ and $\intr$ refer to the closure and the interior respectively.

\begin{proposition}\label{prop:def}
There exists an open neighbourhood $\mathcal{U}$ of the identity in $\Proj(\RP^n)$ such that $\forall K \in \mathcal{V}$ and $\forall \pi \in \mathcal{U}$ we have $\pi \circ F(K)\subset \RP^n\setminus H_0$ and $F^{-1}\circ \pi \circ F(K) \in \mathcal{K}_0^n$.
\end{proposition}

\begin{proof}
Let us introduce the following local parametrization
\begin{eqnarray*}
  \Phi: \GL_n(\R) \times \R^n \times \R^n
    &\to& \Proj(\RP^n)\simeq \PGL_n(\R)\\
  (M,x,\xi)
    &\mapsto&
      \proj\begin{pmatrix}
        1 & \xi^t\\
        x & M
      \end{pmatrix}.
\end{eqnarray*}
Clearly $\Phi$ is a diffeomorphism onto its image which is an open neighbourhood of the identity.
We easily compute for any $\pi=\Phi(M,x,\xi)$ and every $y \in \R^n$ that
\[
  \pi \circ F(y)=[1+\xi^t y,x+My].
\]
So $\pi \circ F(y) \in \RP^n \setminus H_0$ if and only if $1+\xi^ty\neq 0$, and the condition that $\pi \circ F(K)\subset \RP^n\setminus H_0$ amounts to the condition that $-\xi \in \intr(K^\circ)$. Now if $-\xi \in \intr(K^\circ)$ the well-defined set $F^{-1}\circ \pi \circ F(K)$ is automatically a convex body, that contains the zero in its interior if and only if $-x \in \intr(MK)$.
Therefore the open set $\mathcal{U}\coloneqq\Phi\left(\{(M,x,\xi) \mid -\xi \in B_2^\circ \, \, \text{and} \, \, -x \in MB_1\}\right)$ satisfies the conclusion of the proposition.
\end{proof}
Now we will reparametrize this open neighbourhood $\mathcal{U}$ in order to show that the volume product of projective images of $K$ factors through the map $\FF_K$.

\begin{proposition}\label{prop:proj}
There exists an open neighbourhood $\mathcal{W}$ of $(I_n,0,0)$ in $\GL_n(\R) \times \R^n \times \R^n$ and a map $\Psi : \mathcal{W} \to \Proj(\RP^n)$ such that $\Psi(I_n,0,0)=\Id_{\RP^n}$, $\Psi$ is a diffeomorphism onto its image, and the following holds:
\[
  \vp(F^{-1}\circ \Psi(M,x,\xi)\circ F(K))= \FF_K(x,\xi)
\]
for any convex body $K \in \mathcal{V}$ and for any $(M,x,\xi)\in \Psi^{-1}(\mathcal{U})$.
\end{proposition}

Therefore, in order to study the variational properties of the volume product of projective images of a convex body $K \in \mathcal{K}_0^n$, it is enough to study the associated functional $\FF_K$.

\begin{proof}
Consider the map
\begin{eqnarray*}
  \alpha: \GL_n(\R) \times \R^n \times \R^n & \to     & \GL_n(\R) \times \R^n \times \R^n \\
  (N,y,\eta)                                        & \mapsto & \left(N(I_n+y\eta^t),Ny,\eta\right).
\end{eqnarray*}
The map $\alpha$ fixes points of the form $(M,0,0)$ for any $M \in \GL_n(\R)$, and there exists an open neighbourhood $\mathcal{W}$ of $(I_n,0,0)$ in $\GL_n(\R) \times \R^n \times \R^n$ such that the restriction of $\alpha$ to $\mathcal{W}$ is a diffeomorphism onto its image whose inverse is $\alpha^{-1}(M,x,\xi)=\Big(M-x\xi^t,(M-x\xi^t)^{-1}x,\xi\Big)$.
Then define the map
\begin{eqnarray*}
  \Psi: \GL_n(\R) \times \R^n \times \R^n
    &\to&
      \Proj(\RP^n)\\
  (N,y,\eta)
    &\mapsto&
      \proj\begin{pmatrix} 1 & 0\\ 0 & N \end{pmatrix}
        \cdot
      \proj\begin{pmatrix} 1 & 0\\ y & I_n \end{pmatrix}
        \cdot
      \proj\begin{pmatrix} 1 & \eta^t\\ 0 & I_n \end{pmatrix}
  .
\end{eqnarray*}
Because
\[
  \Phi\circ \alpha (N,y,\eta)
  =
  \proj \begin{pmatrix} 1 & \eta^t\\ Ny & N(I_n+y\eta^t) \end{pmatrix},
\]
while
\[
  \proj\begin{pmatrix} 1 & 0\\ 0 & N \end{pmatrix}
    \cdot
  \proj\begin{pmatrix} 1 & 0\\ y & I_n \end{pmatrix}
    \cdot
  \proj\begin{pmatrix} 1 & \eta^t\\ 0 & I_n \end{pmatrix}
  =
  \proj\begin{pmatrix} 1 & 0\\ Ny & N \end{pmatrix}
    \cdot
  \proj\begin{pmatrix} 1 & \eta^t\\ 0 & I_n \end{pmatrix}
  =
  \proj\begin{pmatrix} 1 & \eta^t\\ Ny & N(I_n+y\eta^t) \end{pmatrix},
\]
we have that $\Psi=\Phi\circ \alpha$ which implies that the restriction of $\Psi$ to $\mathcal{W}$ is a diffeomorphism onto its image. By construction $\Psi(I_n,0,0)=\Id_{\RP^n}$.
Still denote by $\Psi : \mathcal{W} \to \Proj(\RP^n)$ the restriction of the map $\Psi$ to $\mathcal{W}$.
Now for any convex body $K \in \mathcal{V}$ and for any $(M,x,\xi)\in \Psi^{-1}(\mathcal{U})$ the following holds:
\begin{align*}
  F^{-1}
    \circ
  \Psi(M,x,\xi)
    \circ
  F(K)
  &=
    F^{-1}
      \circ
    \proj\begin{pmatrix} 1 & 0\\ 0 & M \end{pmatrix}
      \cdot
    \proj\begin{pmatrix} 1 & 0\\ x & I_n \end{pmatrix}
      \cdot
    \proj\begin{pmatrix} 1 & \xi^t\\ 0 & I_n \end{pmatrix} \circ F(K)\\
  &=
    F^{-1}
      \circ
    \proj\begin{pmatrix} 1 & 0\\ 0 & M \end{pmatrix}
      \cdot
    \Psi(I_n,x,\xi)
      \circ
    F(K)\\
  &=
    F^{-1}
      \circ
    \proj\begin{pmatrix} 1 & 0\\ 0 & M \end{pmatrix}
      \circ
    F
      \circ
    F^{-1}
      \circ
    \Psi(I_n,x,\xi)
      \circ
    F(K).
\end{align*}
But for all $y\in \R^n$
\[
  F^{-1} \circ \proj\begin{pmatrix} 1 & 0\\ 0 & M\end{pmatrix} \circ F(y) = My,
\]
which implies using the invariance of the volume product by the linear group that
\[
  \vp(F^{-1}\circ \Psi(M,x,\xi)\circ F(K))= \vp(F^{-1}\circ \Psi(I_n,x,\xi)\circ F(K)).
\]

Now observe that
\[
  F^{-1}\circ \Psi(I_n,x,\xi)\circ F(K)=(K^\circ+\xi)^\circ+x.
\]
Indeed write $F^{-1}\circ \Psi(I_n,x,\xi)\circ F=F^{-1}\circ \Psi(I_n,x,0)\circ F\circ F^{-1}\circ \Psi(I_n,0,\xi)\circ F$. First $F^{-1}\circ \Psi(I_n,x,0)\circ F(y)=x+y$ while $F^{-1}\circ \Psi(I_n,0,\xi)\circ F(y)=y/(1+\xi^t y)$. By a straightforward computation it implies that $F^{-1}\circ \Psi(I_n,0,\xi)\circ F(K)=(K^\circ +\xi)^\circ$.

It follows that $\vp(F^{-1}\circ \Psi(M,x,\xi)\circ F(K))=\vp\left((K^\circ+\xi)^\circ+x\right)$ as claimed.
\end{proof}

\section{A family of critical pentagons}\label{sec:pent}

Fix two numbers $b>0$ and $q \in [-\frac{1}{b},b]$, and set $r = \sqrt{1+q^2} $ and $c= \frac{1+bq}{\sqrt{1+q^2}}$.

We consider the pentagon $P_{q,b}\subset \R^2$ whose vertices are $(-\frac{1}{b},0)$, $(q,\pm r)$ and $(b,\pm c)$. By construction $P = P_{q,b}$ is symmetric with respect to the $x$-axis:

\begin{center}
  \includegraphics[scale=.84]{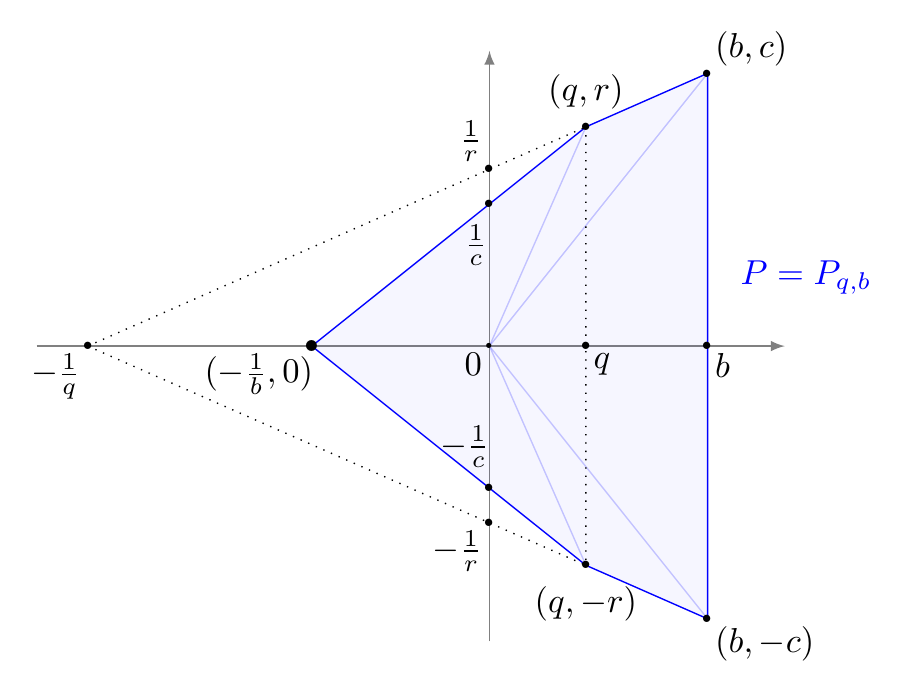}
\end{center}

\begin{lemma}\label{lemma:dualP}
  The dual of $P$ is $-P$.
\end{lemma}

\begin{proof}
  This is a direct consequence of the following general fact for convex polygons: a point $(A,B)$ is a vertex of $P^\circ$ if and only if the line $Ax+By=1$ 
  contains a side of $P$. In this way:
  \begin{itemize}
    \item the line through $(-\frac{1}{q}, 0)$ and $(0, \pm\frac{1}{r})$ containing a side of $P$ corresponds to the vertex $(-q,\pm r)$ of $P^\circ$,
    \item the line through $(-\frac{1}{b}, 0)$ and $(0, \pm\frac{1}{c})$ containing a side of $P$ corresponds to the vertex $(-b,\pm c)$ of $P^\circ$,
    \item the line through $(b,c)$ and $(b,-c)$ containing a side of $P$ corresponds to the vertex $(\frac{1}{b},0)$ of $P^\circ$.
  \end{itemize}
   Therefore $P^\circ$ coincides with $-P$.
\end{proof}

By the previous lemma, if the barycenter of $P$ lies at the origin, the same holds for its dual $-P$, and thus $P$ is a critical point for the volume product on the space of projective images of $P$.

The next lemma gives the exact condition on the pair $(q,b)$ which ensures that $P$ is such a critical point.

\begin{lemma}\label{lemma:criticalP}
  The barycenter of $P$ lies at the origin if and only if
  \[
    f(q,b)\coloneqq q^3b-q^2b^2-q^2+2qb^5+qb+3b^4=1.
  \]
\end{lemma}

\begin{proof}

  Since $P$ is symmetric, its barycenter lies on the $x$-axis, and therefore only the abscissa of the barycenter is relevant. This coordinate can be computed by decomposing the pentagon into six pairwise symmetrical triangles, with one vertex at the origin and one common side with $P$.
  The condition on the barycenter then becomes
  \[
    \frac{(b+b)}{3}\begin{vmatrix}b & b \\ 0 & c\end{vmatrix}
    +
    \frac{(b+q)}{3}\begin{vmatrix}b & q \\ c & r\end{vmatrix}
    +
    \frac{(q-\frac{1}{b})}{3} \begin{vmatrix}q & -\frac{1}{b} \\ r & 0\end{vmatrix}
    = 0.
  \]
  Multiplying by $3b^2r$ and using that $r^2=1+q^2$ and $rc = 1 + bq$ gives the desired condition.
\end{proof}

As concrete examples of pairs $(q,b)$ such that $P$ is critical, i.e. solutions of the equation in Lemma \ref{lemma:criticalP}, we can consider the following pairs:
\begin{itemize}
  \item the two pairs $(-1/\sqrt{2}, \sqrt{2})$ and $(1/\sqrt{2},1/\sqrt{2})$ at which the pentagon degenerates respectively into an equilateral triangle $\Delta$ and its dual $-\Delta$,
  \item the pair $(-\cos(2\pi/5)/\sqrt{\cos(\pi/5)}, \sqrt{\cos(\pi/5)})$ at which the pentagon is regular,
  \item the remarkable pair $(0,1/\sqrt[4]{3})$ which corresponds to a critical saddle point as we will show later.
\end{itemize}

In addition to these few simple concrete examples, the next lemma \ref{lemma:graphP} shows that we have indeed a whole continuous family of critical pentagons $\{P_q\}_{q\in [-1/\sqrt{2},1/\sqrt{2}]}$.

\begin{lemma}\label{lemma:graphP}
  The set
  \[
    \biggl\{ (q,b) \in [-\frac{1}{\sqrt{2}},\frac{1}{\sqrt{2}}] \times [\frac{1}{\sqrt{2}}, \sqrt{2}]\biggm| q^3b-q^2b^2-q^2+2qb^5+qb+3b^4=1\biggr\}
  \]
  is the graph of a decreasing smooth function $b(q)$ in the variable $q$.
\end{lemma}

\begin{center}
  \includegraphics[width=120mm]{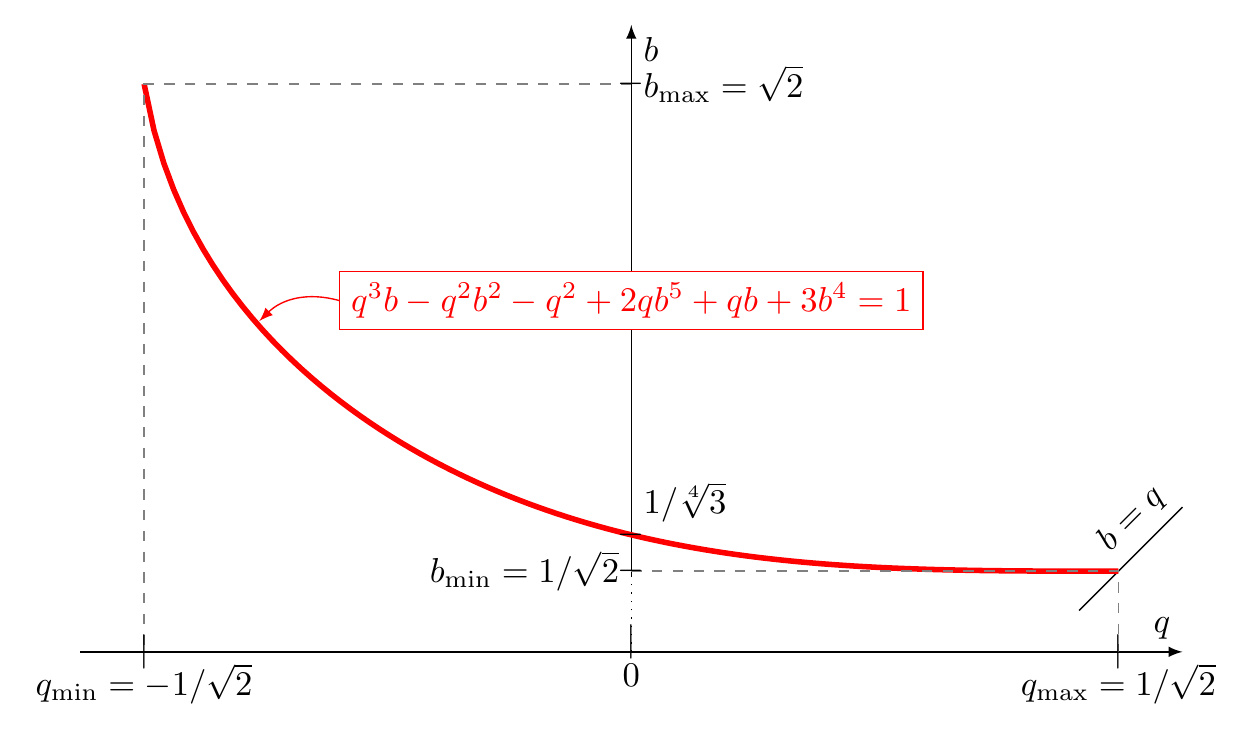}
\end{center}

As mentioned above, the extremal pentagons $P_{\pm1/\sqrt{2}}$ degenerate to equilateral triangles.

\begin{proof}
According to the implicit function theorem, the result follows from elementary computations that both partial derivatives of the function $f(q,b) = q^3b-q^2b^2-q^2+2qb^5+qb+3b^4$ are strictly positive on this domain, except at point $(1/\sqrt{2},1/\sqrt{2})$ where the first partial derivative is zero.

Indeed first observe that $\partial_q f=3bq^2-2(b^2+1)q+2b^5+b$ is a quadratic polynomial in the variable $q$. One can check that its discriminant $-6b^6+b^4-b^2+1$ is always negative for $b \in (1/\sqrt{2},\sqrt{2}]$, and zero at $1/\sqrt{2}$, by studying its first and second derivative.

Next we compute that $\partial_b f=q^3-2bq^2+(10b^4+1)q +12b^3$ and check through the study of $\partial_{bq} f$ that $\partial_b f(q,b)\geq h(b)\coloneqq\partial_b f(-1/\sqrt{2},b)$ which is a polynomial of degree $4$. A straightforward variational study of $h$ ensures its positivity on the desired domain.
\end{proof}

Now we explain how to compute the Hessian of the functional $\FF_{P}$ associated to some pentagon $P$ belonging the above family $\{P_q\}$. For this, we must compute the inertia matrix $I \coloneqq I(P)=\int_P x_i x_jdx$ and the area $|P|$ of such a pentagon. Since these quantities are additive with respect to disjoin union, we can perform these computations by decomposing the pentagon into triangles and using the following lemma.

\begin{lemma}\label{lemma:inertia}
  The inertia matrix of a triangle $\Delta$ with vertices $(0,0)$, $(x_1,y_1)$ and $(x_{2}, y_{2})$ is
  \[
    I(\Delta) = \frac{D}{12}
      \begin{pmatrix}
        I_{xx}(\Delta) & I_{xy}(\Delta)\\
        I_{xy}(\Delta) & I_{yy}(\Delta)
      \end{pmatrix}
  \]
  where

  \[
    \left\{
    \begin{aligned}
      I_{xx}(\Delta) &= x_1^2+x_1x_2+x_2^2\\
      I_{xy}(\Delta) &= x_1y_1 + \frac{x_1y_2+x_2y_1}{2} + x_2y_2\\
      I_{yy}(\Delta) &= y_1^2+y_1y_2+y_2^2
    \end{aligned}
    \right.
  \]
  and
  $
    D = \begin{vmatrix}
        x_1 & x_2 \\
        y_1 & y_2
      \end{vmatrix}
      = |x_1y_2 - x_2y_1|
  $
  is twice the area of the triangle.
\end{lemma}

\begin{proof}
Since we will not need the terms $I_{xy}(\Delta)$ in what follows, we will only do the computations for $I_{xx}(\Delta)$, the case of $I_{yy}(\Delta)$ being identical. An easy change of variables obtained by parametrizing our triangle by the triangle with vertices $(0,0)$, $(1,0)$ and $(0,1)$ gives us
\[
 I_{xx}(\Delta) = \int_0^1\int_0^{1-t} (sx_1+tx_2)^2 D dsdt= \frac{ D }{12} \left[x_1^2 + x_1x_2 + x_2^2\right].
\]
\end{proof}

By symmetry, the inertia matrix of such critical pentagon is diagonal so we only need to compute the terms $I_{xx}(P)$ and $I_{yy}(P)$.
By decomposing as before $P_{q}$ in six pairewise symmetrical triangles we get first that
\begin{equation}\label{area}
  |P_q| = \begin{vmatrix}q & -b^{-1} \\ r & 0\end{vmatrix} + \begin{vmatrix}b & q \\ c & r\end{vmatrix} + \begin{vmatrix}b & b \\ 0 & c\end{vmatrix}.
\end{equation}
Using the same decomposition and Lemma \ref{lemma:inertia} we derive that
  \begin{align}\label{ixx}
      I_{xx}(P) &=&& \frac{1}{6}\begin{vmatrix}q & -b^{-1} \\ r & 0\end{vmatrix}(-b^{-2}+-b^{-1}q+q^2) &&+ \frac{1}{6}\begin{vmatrix}b & q \\ c & r\end{vmatrix}(q^2+qb+b^2) &+& \frac{1}{6}\begin{vmatrix}b & b \\ 0 & c\end{vmatrix}(3b^2)\\[14pt]
      I_{yy}(P) &=&& \frac{1}{6}\begin{vmatrix}q & -b^{-1} \\ r & 0\end{vmatrix}(r^2) &&+ \frac{1}{6}\begin{vmatrix}b & q \\ c & r\end{vmatrix}(r^2+rc+c^2) &+& \frac{1}{6}\begin{vmatrix}b & b \\ 0 & c\end{vmatrix}(c^2)\label{iyy}
    \end{align}

Let us define the functions
\[
  s(q, b) \coloneqq \det\left(|P|^2I_2-16 I(P)^2\right)=\frac19\det(\Hess_{(0,0)}\mathcal{F}_P),
\]
where the second equality follows from equation (\ref{expression det}),
and
\[
  t(q, b) \coloneqq \frac12\tr\left(|P|^2I_2-16I(P)^2\right)=8|P|^2\left(\tr(\cov(\Delta^\circ)\cov(\Delta))-\tr(\cov(P^\circ)\cov(P))\right).
\]

For the pair $(q,b)=(0,-\frac1{\sqrt[4]3})$, we get using \eqref{area},\eqref{ixx} and \eqref{iyy}
\[
 |P|       = \sqrt[4]{3}\Big(          1+\frac{2}{3}\sqrt{3}\Big),\qquad
 I_{xx}(P) = \sqrt[4]{3}\Big(\frac{2}{9}+\frac{1}{6}\sqrt{3}\Big),\qquad
 I_{yy}(P) = \sqrt[4]{3}\Big(\frac{1}{6}+\frac{2}{9}\sqrt{3}\Big)
\]
and therefore
\[
  s(0,-\frac1{\sqrt[4]3})=\frac1{729}(12\sqrt3+17)(4\sqrt3-13)<0.
\]

Hence $\det(\Hess\mathcal{F}_{P_0})$ is negative, and thus $\mathcal{F}_{P_0}$ has a critical saddle point at the origin. This proves item \ref{item_saddle}) of Theorem \ref{th:pentagons}.

To prove item \ref{item_counterexample}) of Theorem \ref{th:pentagons}, it will be enough to compute the derivative of $t(q,b(q))$ at $q_0=-\frac1{\sqrt2}$.
First, at $(q_0,b_0)=(-\frac1{\sqrt2},\sqrt2)$ we have
\begin{align*}
  |P_q|&=\frac{3\sqrt3}2, && \frac{\partial }{\partial q}|P_q|= \frac{\sqrt6}2, &&\frac{\partial }{\partial b}|P_q|=-\frac{\sqrt6}4,\\
  I_{xx}(P_q)&=\frac{3\sqrt3}8, && \frac{\partial }{\partial q}I_{xx}(P_q)= \frac{\sqrt6}2, &&\frac{\partial }{\partial b}I_{xx}(P_q)=-\frac{\sqrt6}4,\\
  I_{yy}(P_q)&=\frac{3\sqrt3}8, && \frac{\partial }{\partial q}I_{yy}(P_q)= -\frac{\sqrt6}8, &&\frac{\partial }{\partial b}I_{yy}(P_q)=-\frac{\sqrt6}{16}.
\end{align*}
On the other hand, one has $\nabla f(q_0,b_0)=\frac{9\sqrt2}4(6,1)$ and thus $\frac{\partial b}{\partial q} (q_0)=-6$. Therefore, using the previous values,
\[
  \left.\frac{\partial t}{\partial q}\right|_{(q_0,b_0)}=-\frac{9\sqrt2}{4},\qquad \left.\frac{\partial t}{\partial b}\right|_{(q_0,b_0)}=\frac{9\sqrt2}{8},\qquad \left.\frac{d}{d q}\right|_{q_0} t(q,b(q))=\left.\frac{\partial t}{\partial q}+\frac{\partial b}{\partial q}\frac{\partial t}{\partial b}\right|_{(q_0,b_0)}=-{9\sqrt2}.
\]

Since $t(q_0,b_0)=0$, this shows that $t(q,b(q))<0$ near $q_0$, thus finishing the proof of Theorem \ref{th:pentagons}.

\begin{center}
  \includegraphics[width=100mm]{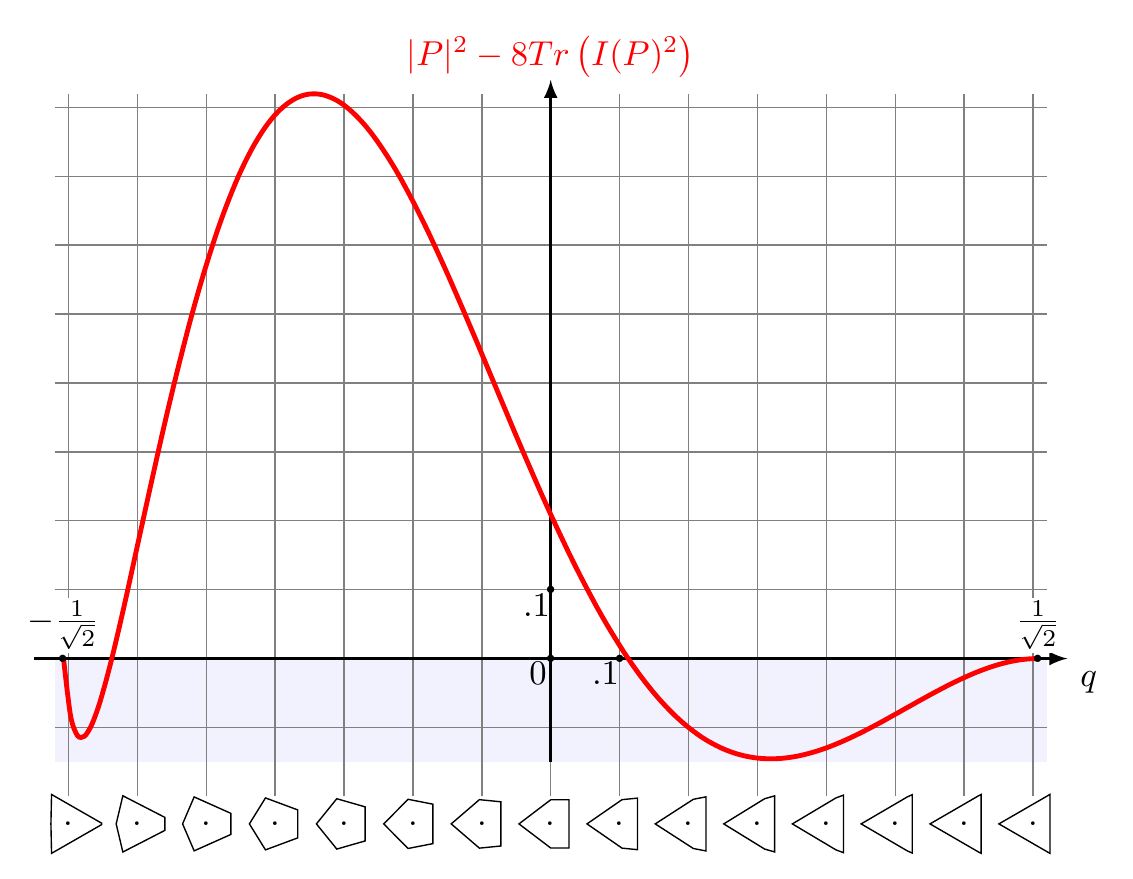}
\end{center}

To complete the picture, we conclude with a numerical analysis of the whole family of critical pentagons $\{P_{q}\colon {q \in [-1/\sqrt{2}, 1/\sqrt{2}]}\}$.

The previous three lemmas (\ref{lemma:dualP}, \ref{lemma:criticalP} and \ref{lemma:graphP}) tell us that for any $q \in [-1/\sqrt{2}, 1/\sqrt{2}]$ we can compute with very high precision (at least $10^{-15}$ using \texttt{python} with \texttt{scipy}) the unique $b \in [1/\sqrt{2}, \sqrt{2}]$ such that $P_{q,b}$ is critical. Then all vertices are computed with precision at least $10^{-14}$, and therefore the polynomial function $|P|^2-8\tr\left(I(P)^2\right)$ is calculated with precision at least $10^{-13}$ to obtain the above curve, which is negative if $\tr[\cov(P_q)\cov(P_q^\circ)] > 1/8$ as $\cov(P_q) =I(P_q)/|P_q|$.



\end{document}